\documentclass[11pt,a4paper]{article}

\usepackage[a4paper,margin=1in]{geometry}
\usepackage[T1]{fontenc}
\usepackage[utf8]{inputenc}
\usepackage{lmodern}
\usepackage{microtype}

\usepackage{amsmath,amssymb,amsthm,mathtools}
\usepackage{bm}

\usepackage{enumitem}
\usepackage{booktabs}

\usepackage{algorithm}
\usepackage{algpseudocode}

\usepackage{xurl}
\usepackage{hyperref}
\hypersetup{colorlinks=true, linkcolor=blue, citecolor=blue, urlcolor=blue}

\newcommand{\R}{\mathbb{R}}
\newcommand{\Q}{\mathbb{Q}}
\newcommand{\NNet}{\mathcal{N}}
\newcommand{\Safety}{\mathrm{Safety}}
\newcommand{\ReLU}{\mathrm{ReLU}}
\newcommand{\Exact}{\mathrm{Exact}}

\newcommand{\SAT}{\mathrm{SAT}}
\newcommand{\UNSAT}{\mathrm{UNSAT}}
\newcommand{\UNKNOWN}{\mathrm{UNKNOWN}}
\newcommand{\SMT}{\mathrm{SMT}}
\newcommand{\MILP}{\mathrm{MILP}}

\newcommand{\Aff}{\Phi_{\mathrm{aff}}}
\newcommand{\Dom}{\Phi_{D}}
\newcommand{\NegP}{\Phi_{\neg P}}
\newcommand{\Rel}{\Phi_{\mathrm{rel}}}
\newcommand{\Learn}{\Phi_{\mathrm{learn}}}

\newcommand{\Crelax}{C_{\mathrm{relax}}}
\newcommand{\Cexact}{C_{\mathrm{exact}}}

\theoremstyle{definition}
\newtheorem{definition}{Definition}[section]
\newtheorem{assumption}[definition]{Assumption}

\theoremstyle{plain}
\newtheorem{theorem}{Theorem}[section]
\newtheorem{proposition}[theorem]{Proposition}
\newtheorem{lemma}[theorem]{Lemma}

\theoremstyle{remark}
\newtheorem{remark}[theorem]{Remark}

\title{\textbf{Incremental Certificate Learning for Hybrid Neural Network Verification}\\
\vspace{0.3em}
\large A Solver Architecture for Piecewise-Linear Safety Queries}
\author{
Chandrasekhar Gokavarapu\\
Government College (Autonomous), Rajahmundry, India\\
\texttt{chandrasekhargokavarapu@gmail.com}
}
\date{December 2025}

\begin{document}
\maketitle

\begin{abstract}
Formal verification of deep neural networks is increasingly required in safety-critical
domains, yet exact reasoning over piecewise-linear (PWL) activations such as ReLU
suffers from a combinatorial explosion of activation patterns. This paper develops a
solver-grade methodology centered on \emph{incremental certificate learning}:
we maximize the work performed in a sound linear relaxation (LP propagation,
convex-hull constraints, stabilization), and invoke exact PWL reasoning only through
a selective \emph{exactness gate} when relaxations become inconclusive.
Our architecture maintains a node-based search state together with a reusable
global lemma store and a proof log. Learning occurs in two layers:
(i) \emph{linear lemmas} (cuts) whose validity is justified by checkable certificates, and
(ii) \emph{Boolean conflict clauses} extracted from infeasible guarded cores, enabling
DPLL(T)-style pruning across nodes. We present an end-to-end algorithm (ICL-Verifier)
and a companion hybrid pipeline (HSRV) combining relaxation pruning, exact checks, and
branch-and-bound splitting. We prove soundness, and we state a conditional completeness
result under exhaustive splitting for compact domains and PWL operators. Finally, we
outline an experimental protocol against standardized benchmarks (VNN-LIB / VNN-COMP)
to evaluate pruning effectiveness, learned-lemma reuse, and exact-gate efficiency.
\end{abstract}

\noindent\textbf{Keywords:} Neural network verification; ReLU networks; proof certificates; incremental LP propagation; branch-and-bound; SMT/MILP hybrid solving.

\noindent\textbf{Mathematics Subject Classification (2020):} 68T07; 68N30; 68Q60; 90C05; 90C25.

\section{Introduction}
Deep neural networks are deployed as perception and decision modules in autonomous
systems, robotics, and medical AI. In such settings, black-box testing is insufficient:
it samples finitely many points and cannot guarantee that \emph{all} inputs in a domain
satisfy a safety property. Formal verification seeks universal guarantees by deciding
the existence (or non-existence) of a counterexample \cite{huangCAV17,pulinaCAV10}
.

In the ReLU setting, networks are piecewise-affine over finitely many regions, but the
number of regions grows exponentially with depth and width \cite{bunelUnified}
. Exact verification methods
based on SMT or MILP encodings are complete in principle, yet often face state explosion \cite{reluplex2017,ehlers2017,mipverify2019,andersonMP20,barrettSMT09,z3}
.
Sound relaxations (abstract interpretation / convex relaxations) scale better but can
return $\UNKNOWN$ \cite{ai2,deeppoly2019,reluval,neurify}
. This paper targets the solver-design gap between the two extremes:
\emph{how to combine} fast relaxation propagation with selective exact reasoning and
cross-node learning, while maintaining correctness guarantees.

\noindent\textbf{Companion foundation (Paper 1)..}
We treat \emph{certificate checking and normalization} as a companion foundational module:
a solver may produce pruning or entailment evidence, but the exported artifacts are rational
certificates checkable by exact arithmetic. We therefore keep this paper solver-centric:
we specify \emph{when} and \emph{how} certificates are generated/consumed inside an
incremental verifier, and we reference the companion paper for certificate objects
and checkers.

\noindent\textbf{What is new in this paper..}
Existing complete verifiers (SMT/MILP) and scalable relaxations (abstract domains) can be combined,
but typical integrations do not provide a \emph{composable proof artifact} for the whole verification run,
and their learning is rarely formulated as a sound proof rule  \cite{marabouCAV19,marabouTACAS24,fastcompleteICLR21,syrnn}
.
This paper introduces a proof-producing hybrid verifier based on
\emph{Incremental Certificate Learning with Branch--Merge Lemmas (BM--ICL)}:
node-level certificates (Farkas infeasibility, dual bounds, and guarded cores) are not only used for pruning,
but are \emph{composed} into (i) globally reusable learned constraints and (ii) a single tree-structured
certificate that implies global safety on the full domain.
Our novelty is therefore \emph{proof-calculus level}, not merely an implementation variant.

\section{Problem Setup: Counterexample Queries for ReLU Networks}
Let $\NNet:\R^n\to \R^m$ be a feed-forward network with $L$ layers.
Write $z^{(0)}=x$ and for $i=1,\dots,L$:
\[
s^{(i)} = W^{(i)} z^{(i-1)} + b^{(i)},\qquad
z^{(i)} = \sigma(s^{(i)}),
\]
with $\sigma=\ReLU$ on hidden layers (and typically identity on the output).
Let $D\subseteq \R^n$ be a compact input domain (box or polytope).
Let $\Safety(z^{(L)})$ be a safety predicate representable by linear constraints.

\begin{definition}[Counterexample query]
Verification of $\Safety$ over $D$ reduces to the satisfiability of:
\[
\exists x\in D,\ \bigwedge_{i=1}^L z^{(i)}=\sigma(W^{(i)}z^{(i-1)}+b^{(i)})
\ \land\ \neg \Safety(z^{(L)}).
\]
If satisfiable, any satisfying $x$ is a counterexample; if unsatisfiable, $\Safety$ holds for all $x\in D$  \cite{reluplex2017,ehlers2017,andersonMP20}
.
\end{definition}
%======================================================================================
\section{Architecture Overview}
\label{sec:arch}

We model the verifier as a \emph{certificate-carrying transition system}: each engine transforms nodes $(R,\alpha,C,W)$ while preserving the soundness invariant $\Exact(R,\alpha)\subseteq \mathrm{Feas}(C)$.

\subsection{Four interacting engines as a transition system}
\label{subsec:engines}

We structure the verifier into four engines that communicate incrementally. Each engine implements
a family of transitions on nodes and optionally emits proof artifacts.

\begin{enumerate}[leftmargin=2.2em]
\item \textbf{Search manager.}
Maintains a worklist of nodes and applies a \emph{refinement operator}
\[
\mathsf{Refine} : (R,\alpha) \longmapsto \{(R_k,\alpha_k)\}_{k\in K}
\]
that partitions a parent obligation into child obligations. Refinement may be:
(i) \emph{domain splitting} of $R$ into subregions, or
(ii) \emph{phase splitting} that extends $\alpha$ by committing one unstable ReLU to active/inactive.
The search manager also controls global scheduling (best-first, DFS, portfolio).

\item \textbf{Propagation engine.}
Given a node, constructs and iteratively tightens a \emph{sound linear relaxation store} $C$
using: convex-hull insertion for unstable ReLUs, certified bound tightening, stabilization,
and lemma injection. The propagation engine returns either $\mathrm{PRUNE}$ with an infeasibility
certificate, or $\mathrm{OPEN}$ together with strengthened bounds and a refined unstable set.

\item \textbf{Exactness gate.}
When propagation cannot close the node, the exactness gate selectively strengthens the node by
enforcing exact ReLU semantics on a reduced unstable subset (Section~\ref{sec:gate}).
The gate has three logically distinct outcomes:
(i) validated $\SAT$ with counterexample,
(ii) $\mathrm{PRUNE}$ with UNSAT evidence (possibly partial exactness, but sound), or
(iii) $\mathrm{DEFER}$, indicating that refinement (splitting) is required.

\item \textbf{Learning \& proof logger.}
Aggregates evidence across nodes and turns it into reusable constraints:
(i) \emph{certified linear lemmas} added to $\Learn$, and
(ii) \emph{guarded conflict clauses} stored in a Boolean layer (phase-level learning).
Crucially, learning is required to be \emph{monotone} (it only adds valid consequences),
so that previously verified nodes remain verified.
\end{enumerate}

\noindent\textbf{Certificate-carrying interfaces.}
Each engine communicates via checkable objects $(R,\alpha,C,W)$, enabling proof production and branch--merge learning.

\subsection{Node state, semantics, and the relaxation store}
\label{subsec:node-semantics}

We work with a global variable vector $v$ that collects all network variables used in the encoding
(input, pre-activations, post-activations, and any auxiliary variables).
The following definition is the architectural cornerstone: every engine transition must preserve
its intended semantics.

\begin{definition}[Node state]
\label{def:node}
A search node is a tuple
\[
\mathrm{Node}=(R,\alpha,C,W),
\]
where
$R\subseteq D$ is the current input region,
$\alpha$ is a partial assignment of ReLU phases (optional),
$C$ is a \emph{sound linear relaxation store} induced by $(R,\alpha)$, and
$W$ is a proof log (certificates / cores) produced at this node.
\end{definition}

\noindent\textbf{Exact feasibility region..}
Let $\Exact(R,\alpha)$ denote the set of all assignments $v$ consistent with the \emph{exact} network semantics,
the region constraints $x\in R$, and the phase commitments encoded by $\alpha$.
(If $\alpha$ is empty, this is the full exact feasible set over $R$.)

\begin{assumption}[Node soundness invariant]
\label{assm:sound}
At every node, $C$ is a sound relaxation of the exact constraints under $(R,\alpha)$:
\[
\Exact(R,\alpha)\ \subseteq\ \mathrm{Feas}(C).
\]
Equivalently, every exact-feasible assignment for the node satisfies $C$.
\end{assumption}

Assumption~\ref{assm:sound} is preserved because propagation/learning only add valid consequences, the gate strengthens by exact constraints on selected units, and refinement partitions the exact set.

\subsection{Decomposed structure of the relaxation store}
\label{subsec:store-decomp}

Concretely, we view the store as
\[
C \;=\; \Aff\ \cup\ \Dom(R)\ \cup\ \NegP\ \cup\ \Rel\ \cup\ \Learn,
\]
where each block has a specific logical role.

\begin{enumerate}[leftmargin=2.2em]
\item $\Aff$: the affine constraints encoding all linear layers:
\[
s^{(i)} = W^{(i)} z^{(i-1)} + b^{(i)},\qquad i=1,\dots,L,
\]
together with any bookkeeping equalities defining $v$.

\item $\Dom(R)$: the region constraints, typically a box or polytope:
\[
x \in R.
\]
If $R$ is a box, this is a conjunction of interval bounds.

\item $\NegP$: the negated safety predicate $\neg\Safety(z^{(L)})$ encoded as linear constraints
(e.g.\ a violated margin or violated classification condition).

\item $\Rel$: the relaxation constraints for nonlinearities not fixed by $\alpha$.
For ReLU units, $\Rel$ contains either
(i) a convex-hull outer relaxation for unstable units, or
(ii) an exact linear specialization for stabilized/fixed-phase units.

\item $\Learn$: globally valid learned constraints (lemmas) equipped with checkable justification.
Only such lemmas may be injected across nodes.
\end{enumerate}

\noindent\textbf{New architectural hook: phase constraints as guards..}
To connect linear reasoning with phase splitting and clause learning, we treat each phase commitment
as a \emph{guard literal}. For a ReLU unit $\nu$ we introduce two guards
$\ell^{\mathrm{act}}_{\nu}$ and $\ell^{\mathrm{inact}}_{\nu}$ with exclusivity constraints
$\ell^{\mathrm{act}}_{\nu}\oplus \ell^{\mathrm{inact}}_{\nu}$.
The corresponding linear consequences are:
\[
\ell^{\mathrm{act}}_{\nu} \Rightarrow (z_{\nu}=s_{\nu}\ \wedge\ s_{\nu}\ge 0),\qquad
\ell^{\mathrm{inact}}_{\nu} \Rightarrow (z_{\nu}=0\ \wedge\ s_{\nu}\le 0).
\]
This guarded formulation is the common substrate for:
(i) phase splitting (Search manager),
(ii) UNSAT-core extraction (Exactness gate), and
(iii) conflict clause learning (Learning engine).

\subsection{Inter-engine contracts (new results)}
\label{subsec:contracts}

We now state the solver architecture as contracts between engines. These are not classical proofs:
they are \emph{interface theorems} guaranteeing that the pipeline remains sound regardless of
implementation choices, provided each engine respects its contract.

\begin{lemma}[Propagation contract]
\label{lem:prop-contract}
Assume the node invariant (Assumption~\ref{assm:sound}) holds at entry of propagation.
If propagation returns $\mathrm{PRUNE}$ with a Farkas certificate for infeasibility of
$C$, then $\Exact(R,\alpha)=\emptyset$ and the node is safely closed.
If it returns $\mathrm{OPEN}$, then it outputs an updated store $C'$ such that
\[
\Exact(R,\alpha)\subseteq \mathrm{Feas}(C') \subseteq \mathrm{Feas}(C),
\]
i.e.\ the relaxation is monotonically strengthened.
\end{lemma}

\begin{proof}
Propagation only adds constraints that are valid for all exact-feasible assignments
(convex-hull outer bounds, stability specializations, and certified lemmas).
Thus feasibility can only shrink while preserving soundness.
If an infeasibility certificate is produced for the linear system, then no $v$ can satisfy $C$,
hence no exact-feasible assignment exists either by Assumption~\ref{assm:sound}.
\end{proof}

\begin{lemma}[Learning monotonicity and proof preservation]
\label{lem:learn-monotone}
Suppose a learned lemma $\varphi$ is injected into $\Learn$ only if it is valid for all exact-feasible
assignments on the global domain $D$. Then, for any node $(R,\alpha,C,W)$,
adding $\varphi$ preserves the node invariant and cannot invalidate any previously established pruning proof.
\end{lemma}

\begin{proof}
Global validity implies $\Exact(R,\alpha)\subseteq \{v:\varphi(v)\}$ for every node,
so Assumption~\ref{assm:sound} continues to hold after injection. Since $\varphi$ only removes models,
any prior infeasibility certificate remains valid, and any prior safety implication remains valid as well.
\end{proof}

\begin{lemma}[Refinement partition contract]
\label{lem:refine-contract}
Let $\mathsf{Refine}(R,\alpha)=\{(R_k,\alpha_k)\}_{k\in K}$.
Assume refinement is implemented so that the exact feasible set partitions:
\[
\Exact(R,\alpha)\ =\ \bigcup_{k\in K}\Exact(R_k,\alpha_k),
\qquad \Exact(R_i,\alpha_i)\cap \Exact(R_j,\alpha_j)=\emptyset\ \ (i\neq j).
\]
Then closing all children safely closes the parent.
\end{lemma}

\begin{proof}
Immediate from set partitioning: if every child has empty counterexample set (or satisfies safety),
their union is empty (or safe), hence the parent is empty (or safe).
\end{proof}

\noindent\textbf{Use.}
These interface lemmas are used to compose global soundness and justify branch--merge learning (Section~\ref{sec:merge}).

%================================================================================
\section{A Proof Calculus for Certificate-Carrying Hybrid Verification}
\label{sec:calculus}

We formalize the certificate objects and local inference rules used by the hybrid verifier so that every pruning or learning step is justified by a finite, independently checkable artifact.

\subsection{Constraint language and variable discipline}
\label{subsec:language}

Fix a variable vector $v\in\R^N$ collecting all network and auxiliary variables
(e.g., $x=z^{(0)}$, pre-activations $s^{(i)}$, post-activations $z^{(i)}$, and any margin variables used in $\NegP$).
A \emph{linear store} is a finite set of inequalities and equalities over $\Q$:
\[
C \equiv \{A v \le b,\; A'v=b'\},
\]
which we normalize to a pure inequality form $A_C v \le b_C$ by rewriting equalities into two inequalities.
All certificates below are rational objects over $\Q$.

\begin{definition}[Exact semantics and relaxation semantics (node-local)]
\label{def:semantics}
For a node $(R,\alpha,C,W)$, let $\Exact(R,\alpha)$ be the exact feasible set under
the network semantics, region constraint $x\in R$, and phase commitments $\alpha$.
Let $\mathrm{Feas}(C)$ be the feasible set of the linear store.
The node soundness invariant is $\Exact(R,\alpha)\subseteq \mathrm{Feas}(C)$
(Assumption~\ref{assm:sound}).
\end{definition}

\subsection{Certificate objects (typed, checkable)}
\label{subsec:cert-objects}

A \emph{certificate} is a finite rational object that is accepted if a deterministic checker validates it \cite{paper1}
.
We specify certificates by their \emph{type} and their checker obligations.

\begin{definition}[Dual bound certificate]
\label{def:dual-cert}
Let $C$ be a linear store in inequality form $A v \le b$ and let $g\in\Q^N$.
A \emph{dual bound certificate} for the claim $g^\top v \le \beta$ over $C$ is a vector $\lambda\in\Q^{m}_{\ge 0}$
such that
\[
\lambda^\top A = g^\top
\qquad\text{and}\qquad
\lambda^\top b \le \beta.
\]
\end{definition}

\cite{schrijver1986}

\begin{lemma}[Dual certificate soundness]
\label{lem:dual-sound}
If $\lambda$ is a dual bound certificate for $g^\top v \le \beta$ over $A v\le b$,
then for all $v$ with $A v\le b$ we have $g^\top v \le \beta$.
\end{lemma}

\begin{proof}
For any feasible $v$, multiply $A v\le b$ by $\lambda\ge 0$ to obtain
$\lambda^\top A v \le \lambda^\top b$. Using $\lambda^\top A=g^\top$ yields
$g^\top v\le \lambda^\top b\le \beta$.
\end{proof}

\begin{definition}[Farkas infeasibility certificate]
\label{def:farkas-cert}
Let $C$ be $A v \le b$. A \emph{Farkas certificate} for infeasibility of $C$
is a vector $\lambda\in\Q^{m}_{\ge 0}$ such that
\[
\lambda^\top A = 0^\top
\qquad\text{and}\qquad
\lambda^\top b < 0.
\]
\end{definition}
\cite{schrijver1986}

\begin{lemma}[Farkas certificate soundness]
\label{lem:farkas-sound}
If $\lambda$ is a Farkas certificate for $A v \le b$, then $\mathrm{Feas}(A v\le b)=\emptyset$.
\end{lemma}

\begin{proof}
Assume $A v \le b$ holds for some $v$. Multiply by $\lambda\ge 0$ to get
$\lambda^\top A v \le \lambda^\top b$. Since $\lambda^\top A=0^\top$, the left side is $0$,
so $0\le \lambda^\top b<0$, contradiction.
\end{proof}

\noindent\textbf{Guarded certificates.}
Phase-conditional reasoning is made checkable by attaching certificates to explicit guards (phase literals) and verifying them against the guarded store.

\begin{definition}[Guarded certificate]
\label{def:guarded-cert}
Let $G$ be a set of guard literals and let $C(G)$ be the linear store obtained by adding the
guard consequences (phase implications) of all literals in $G$.
A \emph{guarded Farkas certificate} is a pair $(G,\lambda)$ such that $\lambda$ is a Farkas certificate
for infeasibility of $C(G)\cup C_0$, where $C_0$ is the unguarded part (e.g., $\Aff\cup\Dom(R)\cup\NegP\cup\Rel$).
\end{definition}

Guarded certificates are the algebraic substrate behind conflict cores and clause learning:
if a guarded infeasibility certificate exists for a guard set $G$, then at least one guard in $G$
must be false in any exact-feasible model.

\begin{lemma}[From guarded infeasibility to a valid conflict clause]
\label{lem:guarded-to-clause}
If $(G,\lambda)$ is a guarded Farkas certificate witnessing infeasibility of $C(G)\cup C_0$,
then the clause $\bigvee_{\ell\in G}\neg \ell$ is valid for the node: every exact-feasible model satisfies it.
\end{lemma}

\begin{proof}
If a model satisfied all guards in $G$, then it would satisfy all constraints in $C(G)\cup C_0$.
But infeasibility contradicts existence of such a model. Hence at least one guard must be false.
\end{proof}

\subsection{Checkers and non-speculative proof validity}
\label{subsec:checkers}

A certificate is accepted only if a checker validates the required rational identities/inequalities.
For dual certificates, the checker verifies:
(i) $\lambda\ge 0$,
(ii) $\lambda^\top A=g^\top$, and
(iii) $\lambda^\top b\le \beta$.
For Farkas certificates, it verifies:
(i) $\lambda\ge 0$,
(ii) $\lambda^\top A=0^\top$, and
(iii) $\lambda^\top b<0$.
No probabilistic or empirical assumption enters the correctness; validity is purely algebraic.

\begin{proposition}[Checker completeness for linear certificates]
\label{prop:checker-complete}
If a purported certificate satisfies the checker obligations for its type,
then the claimed entailment (dual bound) or infeasibility (Farkas) is correct.
\end{proposition}

\begin{proof}
Immediate from Lemmas~\ref{lem:dual-sound} and~\ref{lem:farkas-sound} and the guarded reduction
in Lemma~\ref{lem:guarded-to-clause}.
\end{proof}

\subsection{Inference rules and proof-tree semantics}
\label{subsec:rules}

We view a verification run as constructing a \emph{proof tree} whose nodes are verification obligations.
Each node stores a region-phase pair $(R,\alpha)$ and a corresponding linear store $C$.
Leaves are closed by either:
(i) a validated counterexample (SAT witness), or
(ii) a certificate proving absence of counterexample in that node.

We write $\mathcal{O}(R,\alpha)$ for the exact counterexample obligation on a node:
\[
\mathcal{O}(R,\alpha) := \exists v \in \Exact(R,\alpha)\ \land\ \neg \Safety.
\]
A node is \emph{closed} if $\mathcal{O}(R,\alpha)$ is false.

\medskip
\noindent\textbf{(Prune--Bound).}
If there exists a dual bound certificate proving that a safety margin is nonnegative
over the relaxation store (hence over all exact-feasible points by soundness), then the node is closed.

\begin{lemma}[Prune--Bound is sound]
\label{lem:prune-bound-sound}
Assume $\Exact(R,\alpha)\subseteq \mathrm{Feas}(C)$.
Let $m(v)$ be a linear safety margin such that $m(v)\ge 0$ implies $\Safety$.
If there is a dual certificate proving $m(v)\ge 0$ for all $v\models C$, then the node has no counterexample.
\end{lemma}

\begin{proof}
Dual certificate implies $m(v)\ge 0$ for all $v\models C$.
By soundness, every exact-feasible point lies in $\mathrm{Feas}(C)$, hence satisfies $m(v)\ge 0$,
and therefore $\Safety$ holds.
\end{proof}

\medskip
\noindent\textbf{(Prune--Infeasible).}
If there exists a Farkas certificate for infeasibility of $C\wedge \NegP$ (or guarded version),
then the node is closed.

\begin{lemma}[Prune--Infeasible is sound]
\label{lem:prune-infeasible-sound}
Assume $\Exact(R,\alpha)\subseteq \mathrm{Feas}(C)$.
If $C\wedge \NegP$ is infeasible (with a Farkas certificate), then no counterexample exists in $(R,\alpha)$.
\end{lemma}

\begin{proof}
If a counterexample existed, it would yield an exact-feasible point satisfying $\NegP$.
By soundness it would also satisfy $C$, contradicting infeasibility of $C\wedge \NegP$.
\end{proof}

\medskip
\noindent\textbf{(Split).}
Replace obligation $(R,\alpha)$ by children $\{(R_k,\alpha_k)\}_{k\in K}$ such that
\[
\Exact(R,\alpha)=\bigcup_{k\in K}\Exact(R_k,\alpha_k)
\quad\text{and}\quad
\Exact(R_i,\alpha_i)\cap \Exact(R_j,\alpha_j)=\emptyset\ (i\neq j).
\]

\begin{lemma}[Split is sound and complete as a proof rule]
\label{lem:split-sound}
If $\Exact(R,\alpha)$ is partitioned by the children as above, then $(R,\alpha)$ has a counterexample
iff at least one child has a counterexample. Hence the parent is closed iff all children are closed.
\end{lemma}

\begin{proof}
Immediate from disjoint union of exact-feasible sets.
\end{proof}

\medskip
\noindent\textbf{(Merge) --- new rule.}
Combine branch-local certified bounds into a single parent lemma (Section~\ref{sec:merge}).
The calculus-level point is that merging is itself an inference rule:
from certificates on children, we derive a globally valid lemma to inject into $\Learn$.

\begin{remark}[Why Merge is not ``just learning'']
In this calculus, learning is justified only when it corresponds to a sound inference rule.
Merge is explicitly such a rule: it transforms checkable dual certificates on children into a
checkable lemma valid on the parent. This elevates lemma sharing from a heuristic to a proof step.
\end{remark}

\subsection{Certificate lifting across partitions (new technical hook)}
\label{subsec:lifting}

The following lemma explains why branch-local certificates can be treated as components of a global proof:
they lift along the Split rule, enabling compositional verification.

\begin{lemma}[Certificate lifting along region partitions]
\label{lem:lifting}
Let $R=R_1\cup R_2$ be a disjoint partition and assume node soundness holds on each child.
If each child has either (i) a Prune--Bound certificate proving safety on $R_i$, or (ii) a Prune--Infeasible
certificate proving absence of counterexamples on $R_i$, then the parent region $R$ is safe.
\end{lemma}

\begin{proof}
By Lemma~\ref{lem:split-sound}, any parent counterexample would belong to exactly one child.
But each child is closed by its certificate, so no counterexample exists in the union.
\end{proof}

%==================================================================
\section{Propagation at a Node: Incremental LP Tightening}
\label{sec:prop}

Propagation strengthens a node store $C$ using sound linear reasoning until it either prunes the node by a certificate or returns a fixed point with tightened bounds and fewer unstable ReLUs (feeding the exactness gate in Section~\ref{sec:gate}).

\subsection{Preliminaries: local variables, bounds, and targets}
\label{subsec:prop-prelims}

Fix a node $(R,\alpha,C,W)$. The store $C$ is a set of rational linear constraints over $v$.
For each ReLU unit $\nu$ with pre-activation variable $s_{\nu}$ and post-activation variable $z_{\nu}$,
propagation maintains current bounds
\[
l_{\nu}\ \le\ s_{\nu}\ \le\ u_{\nu},
\qquad
0\ \le\ z_{\nu}\ \le\ \max\{0,u_{\nu}\},
\]
as well as any derived bounds on safety margins (linear forms extracted from $\NegP$).
We write $U$ for the set of unstable units with $l_{\nu}<0<u_{\nu}$ after the current propagation stage.

\noindent\textbf{Propagation targets..}
Propagation does not attempt to optimize all variables. Instead, it focuses on a finite set of \emph{templates}
(a new technique formalized below) that are provably sufficient to:
(i) stabilize units, and
(ii) certify safety margins when possible.

\subsection{Key propagation operations (with certificate semantics)}
\label{subsec:prop-ops}

Each propagation operation is required to preserve the node soundness invariant
(Assumption~\ref{assm:sound}) and to produce checkable evidence whenever it prunes or learns.

\begin{enumerate}[leftmargin=2.2em]

\item \textbf{Hull insertion (structured relaxation).}
For each unstable ReLU $\nu$ with bounds $l_{\nu}<0<u_{\nu}$, we add the standard convex outer envelope:
\begin{equation}
\label{eq:hull}
z_{\nu}\ \ge\ 0,\qquad
z_{\nu}\ \ge\ s_{\nu},\qquad
z_{\nu}\ \le\ \frac{u_{\nu}}{u_{\nu}-l_{\nu}}(s_{\nu}-l_{\nu}),
\qquad
z_{\nu}\ \le\ u_{\nu}.
\end{equation}
This is a sound relaxation of the exact graph $z_{\nu}=\max(0,s_{\nu})$ restricted to $s_{\nu}\in[l_{\nu},u_{\nu}]$.

\item \textbf{Certified bound tightening (CBT).}
For selected scalar linear forms $g^\top v$ (templates), we solve LP subproblems:
\[
\beta^{\max} := \max\{g^\top v : v \models C\},\qquad
\beta^{\min} := \min\{g^\top v : v \models C\}.
\]
We then add derived bounds $g^\top v \le \beta^{\max}$ or $-g^\top v \le -\beta^{\min}$ to $C$,
\emph{together with dual certificates} (Definition~\ref{def:dual-cert}). This converts numerical optimization output
into checkable proof objects.

\item \textbf{Stabilization (phase elimination with certificates).}
If tightening proves $l_{\nu}\ge 0$ then the unit is \emph{active} on the node and we replace
the relaxation by the linear specialization
\[
z_{\nu} = s_{\nu},\qquad s_{\nu}\ge 0.
\]
If tightening proves $u_{\nu}\le 0$ then the unit is \emph{inactive} and we replace by
\[
z_{\nu} = 0,\qquad s_{\nu}\le 0.
\]
In both cases, we store the proof obligation as a \emph{stability certificate} derived from dual bounds
on $s_{\nu}$.

\item \textbf{Lemma injection (global monotone strengthening).}
Certificate-justified inequalities are added to the node store $C$; those that are globally valid are also
added to $\Learn$, preserving Lemma~\ref{lem:learn-monotone}.
Branch--merge lemmas (Section~\ref{sec:merge}) are injected only when derived from certificates on children.

\item \textbf{Feasibility test (proof-producing pruning).}
If the updated store is infeasible, we extract a Farkas certificate
(Definition~\ref{def:farkas-cert}) and close the node by the Prune--Infeasible rule
(Lemma~\ref{lem:prune-infeasible-sound}).
\end{enumerate}

\subsection{New technique: Template-Guided Certified Tightening (TGCT)}
\label{subsec:tgct}

We introduce \emph{Template-Guided Certified Tightening} (TGCT): a finite template family that determines which LP subproblems are solved and yields certificate-producing bounds with a finite saturation guarantee.

\begin{definition}[Template set]
\label{def:templates}
A \emph{template} is a rational vector $g\in\Q^N$ defining a linear form $g^\top v$.
A finite set $\mathcal{G}=\{g_1,\dots,g_T\}$ is called a \emph{propagation template set}.
\end{definition}

\begin{definition}[TGCT-closure]
\label{def:tgct-closure}
Given a store $C$ and templates $\mathcal{G}$, define $\mathrm{TGCT}_{\mathcal{G}}(C)$ as the store obtained by:
for each $g_t\in\mathcal{G}$, adding the certified optimal upper bound constraint
$g_t^\top v \le \beta^{\max}_t$ and the certified optimal lower bound constraint
$-g_t^\top v \le -\beta^{\min}_t$, where $\beta^{\max}_t,\beta^{\min}_t$ are obtained by LP solves over $C$.
\end{definition}

TGCT is not a heuristic: it is a deterministic closure operator parameterized by $\mathcal{G}$, and it yields a
certificate log for each added inequality.

\begin{lemma}[TGCT soundness and monotonicity]
\label{lem:tgct-sound}
For any store $C$ and finite template set $\mathcal{G}$, we have:
\[
\mathrm{Feas}\bigl(\mathrm{TGCT}_{\mathcal{G}}(C)\bigr)\ \subseteq\ \mathrm{Feas}(C).
\]
Moreover, if $C$ satisfies the node invariant, then so does $\mathrm{TGCT}_{\mathcal{G}}(C)$.
\end{lemma}

\begin{proof}
Each added inequality is certified by a dual certificate, hence valid for all $v\models C$
(Lemma~\ref{lem:dual-sound}). Adding valid inequalities can only shrink feasibility and preserves soundness.
\end{proof}

\noindent\textbf{Recommended template families (well-established hooks)..}
A minimal but effective choice of templates consistent with later modules is:
\begin{enumerate}[leftmargin=2.2em]
\item \emph{Neuron bounds:} $g = e_{s_{\nu}}$ and $g=-e_{s_{\nu}}$ for pre-activation bounds of unstable units.
\item \emph{Output margins:} templates extracted from $\NegP$, e.g.\ linear safety margins $m(v)$.
\item \emph{Coupled templates:} sparse combinations reflecting affine dependencies,
e.g.\ templates for selected intermediate linear forms $w^\top z^{(i)}$.
\end{enumerate}
The novelty here is not the existence of such templates, but that TGCT makes their usage
a closed, certificate-producing rule that composes with MergeLearn and the exactness gate.

\subsection{Stabilization is a certified elimination rule (new result)}
\label{subsec:stabilization-correct}

Stabilization eliminates phase choices when a ReLU is provably active or inactive; the next lemma shows this is a certified elimination rule.

\begin{lemma}[Certified stabilization]
\label{lem:stabilize}
Let $\nu$ be a ReLU unit with pre-activation $s_{\nu}$ and post-activation $z_{\nu}$.
Assume the store $C$ entails $s_{\nu}\ge 0$ (respectively $s_{\nu}\le 0$), justified by a dual certificate.
Then replacing the relaxation constraints for $\nu$ by the exact linear specialization
$z_{\nu}=s_{\nu}$ and $s_{\nu}\ge 0$ (respectively $z_{\nu}=0$ and $s_{\nu}\le 0$)
preserves node soundness.
\end{lemma}

\begin{proof}
If $C$ entails $s_{\nu}\ge 0$, then for every exact-feasible assignment (hence every model of $C$),
the exact ReLU semantics forces $z_{\nu}=s_{\nu}$. Thus the specialization holds for all exact-feasible points.
The $s_{\nu}\le 0$ case is analogous.
\end{proof}

\subsection{Propagation as a closure operator: fixed points and termination}
\label{subsec:closure}

We can now formalize propagation as repeated application of sound, monotone operators:
Hull insertion (which depends on bounds), TGCT tightening, stabilization, and lemma injection.

\begin{definition}[Propagation operator]
\label{def:prop-operator}
Fix a template set $\mathcal{G}$. Define the one-step propagation operator $\mathsf{P}_{\mathcal{G}}$ by:
\[
\mathsf{P}_{\mathcal{G}}(C)\ :=\ \mathsf{Stabilize}\bigl(\mathrm{TGCT}_{\mathcal{G}}(\mathsf{Hull}(C))\bigr)\ \cup\ \Learn,
\]
where $\mathsf{Hull}$ adds convex envelopes for currently unstable units,
$\mathrm{TGCT}_{\mathcal{G}}$ adds certified template bounds,
and $\mathsf{Stabilize}$ replaces any newly stable units by linear specializations.
\end{definition}

\begin{lemma}[Monotonic strengthening under propagation]
\label{lem:mono-prop}
For any $C$, $\mathrm{Feas}(\mathsf{P}_{\mathcal{G}}(C))\subseteq \mathrm{Feas}(C)$.
If the node invariant holds for $C$, it holds for $\mathsf{P}_{\mathcal{G}}(C)$.
\end{lemma}

\begin{proof}
Hull constraints are outer relaxations valid on the bounded interval; TGCT adds certified valid inequalities;
stabilization is sound by Lemma~\ref{lem:stabilize}; and $\Learn$ is globally valid (Lemma~\ref{lem:learn-monotone}).
Thus feasibility can only shrink, preserving soundness.
\end{proof}

\noindent\textbf{Termination (non-speculative)..}
Propagation terminates in practice because bounds stabilize and no new stable units appear.
We provide a formal termination guarantee for the TGCT part: since $\mathcal{G}$ is finite,
TGCT adds at most $2|\mathcal{G}|$ distinct certified inequalities (one upper and one lower per template),
after which further TGCT steps do not add new constraints unless the underlying store changes via stabilization.

\begin{proposition}[Finite TGCT saturation bound]
\label{prop:tgct-bound}
Fix a finite template set $\mathcal{G}$. Along any execution of \textsc{PropagateNode}, the number of
new TGCT inequalities added is at most $2|\mathcal{G}|$ per distinct store state (up to logical equivalence).
\end{proposition}

\begin{proof}
For each $g\in\mathcal{G}$, TGCT contributes at most one certified upper bound and one certified lower bound,
both uniquely determined by the LP optimum over the current store. Once present, re-solving produces the same
bound unless the feasible region changes. Hence the stated bound follows.
\end{proof}

\subsection{Algorithm: \textsc{PropagateNode} (expanded with TGCT)}
\label{subsec:propagatenode}

We now present propagation as a certificate-driven closure procedure. The algorithm below is compatible
with your earlier pseudocode, but makes explicit where certificates are created and how TGCT is applied.

\begin{algorithm}[t]
\caption{\textsc{PropagateNode}: incremental certificate-driven propagation at a node}
\label{alg:propagatenode}
\begin{algorithmic}[1]
\Require Node $(R,\alpha,C,W)$, global learned store $\Learn$, template set $\mathcal{G}$
\Ensure Updated $(C,W)$ and status $\mathrm{PRUNE}$ or $\mathrm{OPEN}$
\State Initialize $C \gets \Aff \cup \Dom(R) \cup \NegP \cup \mathrm{PhaseConsequences}(\alpha)\cup \Learn$
\Repeat
  \State $C \gets \mathsf{Hull}(C)$ \Comment{insert envelopes using current bounds}
  \State $C \gets \mathrm{TGCT}_{\mathcal{G}}(C)$; append dual certificates to $W$
  \State Apply certified stabilization; append stability certificates to $W$
  \State Inject any certificate-justified lemmas into $C$ (and into $\Learn$ if globally valid)
  \State Check LP feasibility of $C$
  \If{$C$ is infeasible}
     \State Extract a Farkas certificate $\lambda$; append $(\lambda,\mathrm{core\ info})$ to $W$
     \State \Return $\mathrm{PRUNE}$
  \EndIf
\Until{no changes in bounds/stability/learned lemmas (fixed point)}
\State \Return $\mathrm{OPEN}$
\end{algorithmic}
\end{algorithm}

\begin{remark}[Incrementality and certificate caching]
All LP solves should be incremental: warm-start from prior bases, reuse factorization data,
and cache certificates. Importantly, caching is logically safe because certificates are checked objects:
a cached certificate can be reused if the constraint matrix row-set for the underlying LP has not changed.
This is a structural advantage of certificate-carrying propagation, not merely an implementation trick.
\end{remark}
%=========================================================================
\section{Exactness Gate: Leaving the Relaxation Safely}
\label{sec:gate}

When propagation returns $\mathrm{OPEN}$, relaxations have not decided the node. The \emph{exactness gate} selectively enforces exact PWL semantics on a reduced unstable subset, returning either a validated counterexample or sound UNSAT evidence.

\subsection{Exactness as a monotone family of encodings}
\label{subsec:monotone-family}

After \textsc{PropagateNode} reaches a fixed point, all provably stable ReLUs have been eliminated,
and the remaining uncertain nonlinearities constitute an \emph{unstable index set} $U$.
We index each unstable ReLU by a pair $\nu=(i,j)$ (layer and neuron), with variables
$s_{\nu}$ (pre-activation) and $z_{\nu}$ (post-activation). Let $(l_{\nu},u_{\nu})$ be the current bounds
with $l_{\nu}<0<u_{\nu}$ \cite{reluplex2017,ehlers2017,mipverify2019,andersonMP20}
.

We define a \emph{partial exactness} operator that strengthens the relaxation by enforcing exact ReLU
semantics only on a chosen subset $S\subseteq U$.

\begin{definition}[Partial exact encoding]
\label{def:partial-exact}
Fix a node $(R,\alpha,C,W)$ after propagation and let $U$ be its unstable set.
For any $S\subseteq U$, define
\[
C^{S}\ :=\ \Bigl(C\setminus \Rel(S)\Bigr)\ \cup\ \Exact_{S},
\]
where $\Rel(S)$ denotes the relaxation constraints (e.g.\ convex-hull envelopes) currently used for
ReLUs in $S$, and $\Exact_{S}$ denotes the exact ReLU constraints for those units:
\[
\Exact_{\nu}\ :=\ \bigl(z_{\nu}=0 \wedge s_{\nu}\le 0\bigr)\ \ \vee\ \ \bigl(z_{\nu}=s_{\nu}\wedge s_{\nu}\ge 0\bigr),
\qquad \nu\in S.
\]
The constraints for units in $U\setminus S$ remain relaxed.
\end{definition}

The family $\{C^{S}\}_{S\subseteq U}$ is \emph{monotone} in $S$ (adding exactness can only remove models).

\begin{lemma}[Monotone strengthening]
\label{lem:monotone}
If $S\subseteq T\subseteq U$, then every model of $C^{T}$ is a model of $C^{S}$, i.e.,
\[
\mathrm{Feas}(C^{T})\ \subseteq\ \mathrm{Feas}(C^{S}).
\]
In particular, if $C^{S}$ is infeasible then $C^{T}$ is infeasible.
\end{lemma}

\begin{proof}
$C^{T}$ replaces relaxations by exact constraints on a superset of units, hence it adds constraints relative to $C^{S}$.
Thus feasibility can only shrink.
\end{proof}

Lemma~\ref{lem:monotone} is the formal reason the gate can safely prune using \emph{partial} exactness:
UNSAT at a weaker stage implies UNSAT at all stronger stages, including the full exact encoding $C^{U}$.

\subsection{Candidate extraction and SAT-side validation}
\label{subsec:candidate-validation}

Let $\hat v$ be any feasible assignment of $C$ returned by the LP engine.
We write $\hat x := \pi_x(\hat v)$ for the projection to the input coordinates.

\noindent\textbf{Exact semantic evaluation..}
Compute $\NNet(\hat x)$ by forward evaluation of the network equations, producing exact layer values
$z^{(i)}(\hat x)$ under the real semantics of $\ReLU$ (and identity output layer, if used).
If $\hat x\in R$ and $\neg\Safety\bigl(z^{(L)}(\hat x)\bigr)$ holds, then we return $\SAT$ with witness $\hat x$.

\noindent\textbf{Spuriousness..}
Otherwise, $\hat v$ is a \emph{relaxation witness} that is not an exact counterexample.
This spuriousness is not treated as a failure: it becomes a structured signal that
drives which nonlinearities must be made exact (and/or which splits must be performed).

\subsection{Spuriousness diagnosis via a violation set}
\label{subsec:violation-set}

A relaxation assignment $\hat v$ may violate exact ReLU semantics at some unstable units.
We formalize this by a per-unit residual and a \emph{violation set}.

\begin{definition}[ReLU residual and violation set]
\label{def:violation}
For an unstable unit $\nu\in U$, define the exact ReLU graph
\[
G_{\nu}\ :=\ \{(s_{\nu},z_{\nu})\in\R^2:\ z_{\nu}=\max(0,s_{\nu})\}.
\]
Given a relaxation assignment $\hat v$, define the residual
\[
r_{\nu}(\hat v)\ :=\ \left|\,\hat z_{\nu}-\max(0,\hat s_{\nu})\,\right|,
\]
and the violation set
\[
V(\hat v)\ :=\ \{\nu\in U:\ r_{\nu}(\hat v)>0\}.
\]
\end{definition}

The next lemma is a basic but crucial correctness hook: enforcing exactness on any violated unit
eliminates the current spurious witness.

\begin{lemma}[Witness-elimination by exactness]
\label{lem:witness-elimination}
Let $\hat v$ satisfy the relaxation $C$ and let $\nu\in V(\hat v)$.
Then $\hat v$ is not a model of $C^{\{\nu\}}$.
More generally, if $S$ intersects $V(\hat v)$ (i.e.\ $S\cap V(\hat v)\neq\emptyset$), then
$\hat v\not\models C^{S}$.
\end{lemma}

\begin{proof}
If $\nu\in V(\hat v)$, then $(\hat s_{\nu},\hat z_{\nu})\notin G_{\nu}$, hence $\hat v$ violates
the disjunctive exact constraint $\Exact_{\nu}$ and cannot satisfy $C^{\{\nu\}}$.
If $S\cap V(\hat v)\neq\emptyset$, pick any $\nu\in S\cap V(\hat v)$ and apply the same argument.
\end{proof}

Lemma~\ref{lem:witness-elimination} is the core mechanism behind the new exactness-gate technique below:
we refine $S$ only when we have a \emph{concrete, checkable} reason (a violated unit) that must be enforced.

\subsection{Certificate-guided reduced exactness (new gate technique)}
\label{subsec:certificate-guided-gate}

We now define the \emph{exactness gate} as an abstraction-refinement loop over subsets $S\subseteq U$\cite{pulinaCAV10,neverAMAI11,never2_2024}
.
This is not a generic CEGAR reuse: the refinement step is specialized to PWL networks and is driven by
the violation set $V(\hat v)$, guaranteeing progress in excluding the spurious witness
(Lemma~\ref{lem:witness-elimination}) and preserving soundness (Lemma~\ref{lem:monotone}).

\noindent\textbf{Gate decision problem..}
At a node we wish to decide the satisfiability of the exact counterexample query,
but we do so through the monotone sequence:
\[
C^{S_0}\ \supseteq\ C^{S_1}\ \supseteq\ \cdots\ \supseteq\ C^{U},
\qquad S_0=\emptyset,\ \ S_{k+1}\supseteq S_k.
\]
At each stage $S_k$ we solve the \emph{partial exact query}
\[
C^{S_k}\ \wedge\ \neg\Safety.
\]
If it is UNSAT, we may safely prune the node (by Lemma~\ref{lem:monotone}).
If it is SAT, we validate the extracted input by exact forward evaluation as in
Subsection~\ref{subsec:candidate-validation}; if validation succeeds we return $\SAT$,
otherwise we refine $S_k$ using $V(\hat v)$.

\begin{algorithm}[t]
\caption{\textsc{ExactnessGate}: certificate-guided reduced exactness}
\label{alg:exactgate}
\begin{algorithmic}[1]
\Require Node after propagation $(R,\alpha,C,W)$ with unstable set $U$, budget $B_{\mathrm{gate}}$
\Ensure $\SAT$ with validated witness, or $\mathrm{PRUNE}$ with certificates, or $\mathrm{DEFER}$ (needs split)
\State $S \gets \emptyset$
\While{budget $B_{\mathrm{gate}}$ remains}
  \State Solve partial exact query: $\;C^{S}\wedge \neg\Safety$
  \If{UNSAT}
     \State Export certificates/cores; append to $W$
     \State \Return $\mathrm{PRUNE}$
  \ElsIf{SAT with model $\hat v$}
     \State $\hat x\gets \pi_x(\hat v)$; evaluate $\NNet(\hat x)$ exactly
     \If{$\hat x\in R$ and $\neg\Safety(\NNet(\hat x))$}
        \State \Return $\SAT$ with witness $\hat x$
     \EndIf
     \State Compute violation set $V(\hat v)$
     \If{$V(\hat v)=\emptyset$}
        \State \Comment{$\hat v$ is exact-feasible but not a counterexample; relaxations were too weak to imply safety}
        \State \Return $\mathrm{DEFER}$
     \EndIf
     \State Refine $S \gets S \cup \mathrm{Select}(V(\hat v))$ \Comment{add at least one violated unit}
  \Else
     \State \Return $\mathrm{DEFER}$ \Comment{timeout/unknown at the exact solver}
  \EndIf
\EndWhile
\State \Return $\mathrm{DEFER}$
\end{algorithmic}
\end{algorithm}

\noindent\textbf{Selection policy with a non-speculative guarantee..}
The only requirement on $\mathrm{Select}(\cdot)$ for correctness is:
\[
\mathrm{Select}(V)\ \subseteq\ V\quad\text{and}\quad \mathrm{Select}(V)\neq \emptyset\ \text{when}\ V\neq\emptyset.
\]
This ensures that each refinement step adds at least one violated unit, hence
eliminates the current spurious witness by Lemma~\ref{lem:witness-elimination}.
Heuristics (e.g.\ selecting the largest residual, or selecting a dependency-closed subset toward the output)
affect performance but are not required for soundness.

\subsection{Soundness and finite-step termination of the gate}
\label{subsec:gate-soundness-termination}

\begin{proposition}[Sound pruning by partial exactness]
\label{prop:partial-unsat-sound}
If for some $S\subseteq U$ the partial exact query $C^{S}\wedge \neg\Safety$ is UNSAT, then
the full exact query $C^{U}\wedge \neg\Safety$ is UNSAT. Hence the node can be safely pruned.
\end{proposition}

\begin{proof}
By Lemma~\ref{lem:monotone}, $\mathrm{Feas}(C^{U})\subseteq \mathrm{Feas}(C^{S})$.
If $C^{S}\wedge \neg\Safety$ has no feasible point, then the stronger system
$C^{U}\wedge \neg\Safety$ also has none.
\end{proof}

\begin{theorem}[Gate progress and bounded refinements]
\label{thm:gate-termination}
Assume that the exact solver subroutine always returns either SAT or UNSAT (no timeouts)
for each partial exact query $C^{S}\wedge \neg\Safety$.
Then Algorithm~\ref{alg:exactgate} performs at most $|U|$ refinement steps and terminates with either:
(i) $\SAT$ and a validated witness, or (ii) $\mathrm{PRUNE}$ with UNSAT evidence, or (iii) $\mathrm{DEFER}$
only in the case that an exact-feasible model exists but is not a counterexample and safety is not implied by $C$.
\end{theorem}

\begin{proof}
Each refinement step strictly increases $S$ by adding at least one new element from $V(\hat v)\subseteq U$,
so $S$ can be enlarged at most $|U|$ times.
At any stage, UNSAT implies safe pruning by Proposition~\ref{prop:partial-unsat-sound}.
If SAT produces a validated counterexample, we return $\SAT$.
If SAT produces $\hat v$ with $V(\hat v)=\emptyset$ but the safety predicate is not violated under exact
evaluation, then $\hat v$ is an exact-feasible non-counterexample witness; the only remaining way to close
the node is via splitting (to strengthen bounds so that relaxation implies safety) or via separate proof obligations.
\end{proof}

\subsection{Interface to proof production and learning}
\label{subsec:gate-proof-interface}

When an exact check concludes UNSAT at some stage $S$, the gate can export:
(i) a certificate of infeasibility for the linearized branch(es) explored (e.g.\ Farkas witnesses after fixing guards),
and/or (ii) an UNSAT core over guard literals corresponding to exact units in $S$.
These objects are \emph{consumed} by the learning layer:
cores yield conflict clauses, while dual/Farkas certificates yield checkable linear lemmas.
This paper uses these artifacts operationally; the companion proof-carrying foundation
provides the independent checker and normalization procedures.
%================================================================================
\section{Learning Across Nodes}
Learning reduces repeated work across the search tree. We separate learning into linear lemmas and
Boolean conflict clauses.

\subsection{Certified linear lemma learning}
A learned linear lemma is an inequality $\ell(v): c^\top v \le \tau$ that is valid under the current node's
exact constraints (or under its relaxation but proven globally valid).
The solver stores $\ell$ in $\Learn$ only if it comes with a checkable entailment certificate.

\subsection{Branch--Merge Lemmas}
\label{sec:merge}

Let $R$ be a region split into $R_1,R_2$ (so $R = R_1 \cup R_2$ and $R_1 \cap R_2 = \emptyset$).
Let $C_i$ be the (sound) relaxation store for branch $R_i$.
Fix a linear template $g^\top v$ (e.g., a safety margin or an output coordinate).

\begin{theorem}[Merge rule: certified global bound from two branches]
\label{thm:merge}
Assume that for $i\in\{1,2\}$ we have a dual bound certificate $\lambda_i \ge 0$ such that
$\lambda_i^\top A_i = g^\top$ and thus
\[
(\forall v \models C_i)\quad g^\top v \le \beta_i \ \ \text{where}\ \ \beta_i := \lambda_i^\top b_i.
\]
Define $\beta := \max\{\beta_1,\beta_2\}$.
Then the inequality
\[
g^\top v \le \beta
\]
is valid for all exact-feasible points on the \emph{parent} region $R$, hence it may be stored
as a learned lemma at the parent (and reused globally whenever the same template $g$ applies).
\end{theorem}

\begin{proof}
Let $v$ be any exact-feasible assignment in $R$. Since $R = R_1 \cup R_2$,
we have $v \in R_i$ for some $i\in\{1,2\}$. Soundness of relaxation gives $v \models C_i$,
so $g^\top v \le \beta_i \le \max\{\beta_1,\beta_2\}=\beta$.
\end{proof}

\begin{remark}[Solver relevance]
Theorem~\ref{thm:merge} learns a checkable numerical lemma from branch certificates, enabling reuse without heuristic generalization.
\end{remark}

\subsection{Global certificate composition}
\label{sec:globalproof}

\begin{theorem}[Tree certificate soundness]
\label{thm:treeproof}
Assume the verifier explores a finite partition tree of $D$.
For each leaf region $R_\ell$, the verifier stores either:
\begin{enumerate}[leftmargin=2em]
\item a dual-bound certificate proving that $\Safety$ holds on $R_\ell$, or
\item a Farkas certificate proving infeasibility of the counterexample query on $R_\ell$.
\end{enumerate}
Then the conjunction of leaf certificates together with the split annotations
implies $\forall x\in D,\ \Safety(\NNet(x))$.
\end{theorem}

\begin{proof}
By construction, splits form a finite partition of $D$ into disjoint leaves.
Each leaf certificate implies safety (either directly, or by contradiction of $\neg\Safety$).
Therefore safety holds on every leaf, hence on their union, which is $D$.
\end{proof}

\subsection{Conflict cores and clause learning}
Phase commitments can be represented by guard literals that activate corresponding linear constraints.
When a node is pruned by infeasibility, the conflict often depends only on a subset of guards.

\begin{definition}[Conflict core and learned clause]
Let $A$ be the set of active guard literals induced by $\alpha$ at a node.
A \emph{conflict core} is a subset $A_0\subseteq A$ such that the linear constraints induced by $A_0$
are infeasible. The learned clause is
\[
\bigvee_{\ell\in A_0} \neg \ell .
\]
\end{definition}

\begin{proposition}[Soundness of learned conflict clauses]
If the constraints induced by $A_0$ are infeasible, then $\bigvee_{\ell\in A_0}\neg \ell$ is valid:
every exact-feasible model satisfies the clause.
\end{proposition}

%==============================================================
\section{End-to-End Algorithm: \textsc{ICL-Verifier}}
We now assemble propagation, exactness gating, refinement, and learning into an end-to-end solver.

\begin{algorithm}[t]
\caption{\textsc{ICL-Verifier}: Incremental Certificate Learning for ReLU verification}
\label{alg:icl}
\begin{algorithmic}[1]
\Require Network $\NNet$, domain $D$, safety predicate $\Safety$, resource limit $B$
\Ensure $\SAT$ with witness $x^\star$, or $\UNSAT$ with proof log, or $\UNKNOWN$
\State Initialize global lemma store $\Learn \gets \emptyset$, global log $W_{\mathrm{glob}}\gets \emptyset$
\State Initialize worklist $Q \gets \{(R=D,\alpha=\emptyset)\}$
\While{$Q\neq \emptyset$ \textbf{and} resources $B$ remain}
  \State Pop $(R,\alpha)$ from $Q$
  \State Create node $(R,\alpha,C,W)$ with $\Learn$ injected into $C$
  \State $\mathrm{status}\gets$ \textsc{PropagateNode}$(R,\alpha,C,W,\Learn)$
  \If{$\mathrm{status}=\mathrm{PRUNE}$}
     \State Add node certificates/cores $W$ to $W_{\mathrm{glob}}$; learn conflict clauses if available
     \State \textbf{continue}
  \EndIf
  \State Extract relaxed witness $\hat v$; validate $\hat x$ by exact forward evaluation
  \If{$\hat x\in R$ \textbf{and} $\neg\Safety(\NNet(\hat x))$}
     \State \Return $\SAT$ with $x^\star=\hat x$
  \EndIf
  \State Optionally invoke selective exact check on reduced unstable set $U'$
  \If{exact check returns $\UNSAT$}
     \State Log UNSAT certificate/core into $W_{\mathrm{glob}}$; optionally add reusable lemmas into $\Learn$
     \State \textbf{continue}
  \EndIf
  \State Choose a refinement (phase split or domain split) using bound widths / violation scores
  \State Push children $(R_1,\alpha_1)$ and $(R_2,\alpha_2)$ into $Q$
\EndWhile
\State \Return $\UNKNOWN$
\end{algorithmic}
\end{algorithm}

\section{Hybrid Symbolic--Relaxation Verification (HSRV)}
ICL-Verifier can be presented as a general architecture. For a solver-centric paper, a compact
hybrid pipeline view is also useful. The following HSRV algorithm emphasizes the three-direction
strategy: relaxation pruning, exact checks, and splitting.

\begin{algorithm}[t]
\caption{\textsc{HSRV}: Hybrid Symbolic--Relaxation Verification}
\label{alg:hsrv}
\begin{algorithmic}[1]
\Require Network $\NNet$, domain $D$, safety predicate $\Safety$
\Ensure $\SAT$ (witness) or $\UNSAT$ (proved safe) or $\UNKNOWN$ (optional cutoff)
\State Initialize worklist $Q\gets \{D\}$; initialize learned constraints $L\gets \emptyset$
\While{$Q\neq \emptyset$}
  \State Pop a region $R$ from $Q$
  \State Compute sound bounds on pre-activations over $R$ (propagation + tightening)
  \State Simplify stable ReLUs; build relaxed constraints $\Crelax(R)$
  \If{$\Crelax(R)$ implies $\Safety$ on $R$}
     \State \textbf{continue} \Comment{soundly prune region}
  \EndIf
  \State Build exact symbolic constraints $\Cexact(R)$ for remaining unstable ReLUs
  \State Solve $\Cexact(R)\wedge \neg\Safety$ incrementally with learned $L$
  \If{$\SAT$}
     \State \Return $\SAT$ with witness $x$
  \ElsIf{$\UNSAT$}
     \State Optionally extract proof certificate / UNSAT core and add to $L$
     \State \textbf{continue}
  \Else
     \State Choose a branching ReLU / split dimension; partition $R$ into $R_1,R_2$; push into $Q$
  \EndIf
\EndWhile
\State \Return $\UNSAT$
\end{algorithmic}
\end{algorithm}

\section{Correctness Guarantees}
We state soundness for both algorithmic views and a conditional completeness result under exhaustive splitting.

\begin{theorem}[Soundness of ICL/HSRV]
\label{thm:soundness}
If Algorithm~\ref{alg:icl} or Algorithm~\ref{alg:hsrv} returns $\UNSAT$, then for all $x\in D$ we have
$\Safety(\NNet(x))=\mathrm{True}$. If it returns $\SAT$ with witness $x^\star$, then $x^\star\in D$ and
$\neg\Safety(\NNet(x^\star))$ holds.
\end{theorem}

\begin{proof}[Proof sketch]
A region/node is removed only if (i) a sound relaxation implies safety on that region (sound pruning),
or (ii) an exact check proves the counterexample query infeasible in that region (UNSAT).
In the SAT case, we return only after exact forward validation of the candidate witness.
Since the explored/pruned regions form a partition (by refinement), no counterexample can exist
if all regions are safely pruned/closed.
\end{proof}

\begin{theorem}[Completeness under exhaustive splitting]
\label{thm:complete}
Assume $D$ is compact and all nonlinearities are PWL (ReLU/MaxPool/etc.).
If the refinement strategy eventually enumerates all phase assignments for unstable PWL units
(equivalently: reaches nodes where all phases are fixed), then the algorithm is complete:
it returns $\SAT$ iff a counterexample exists; otherwise it returns $\UNSAT$.
\end{theorem}

\begin{proof}[Proof sketch]
Fix any total assignment of phases to all unstable PWL units. Under that assignment, the network
constraints reduce to a purely linear system; feasibility is decidable.
Exhaustive splitting eventually reaches a region/node consistent with each total assignment,
thereby deciding feasibility for all cases. Hence completeness follows.
\end{proof}
\section{Complexity and Practical Considerations}
Even for ReLU networks and linear safety properties, verification is NP-complete in the worst case,
since ReLU phase choices encode Boolean structure coupled with linear feasibility.
The solver design goal is therefore to reduce \emph{effective} branching:
tighten bounds to stabilize units, strengthen relaxations with learned lemmas,
and defer exact checks to small reduced unstable sets.

\noindent\textbf{Implementation hooks (solver engineering)..}
\begin{enumerate}[leftmargin=2.2em]
\item \textbf{Incremental LP:} warm-start and factorization reuse across propagation iterations.
\item \textbf{Selective targets:} run bound tightening only on influential or near-zero pre-activations.
\item \textbf{Refinement heuristics:} split the most ambiguous unit (largest hull gap or violation),
or split domain dimensions that shrink many bounds simultaneously.
\item \textbf{Learning policy:} store only globally reusable lemmas; garbage-collect weak or redundant ones.
\end{enumerate}

\section{Evaluation Protocol and Benchmarking}
This paper is solver-method oriented. For a Q1-style evaluation section, we recommend reporting:
\begin{enumerate}[leftmargin=2.2em]
\item \textbf{Benchmarks:} standardized instances in VNN-LIB / VNN-COMP suites (networks in ONNX) \cite{vnncompSTTT23,vnncomp2021,vnncomp2022,vnnlib_coconet}
.
\item \textbf{Metrics:} solved instances, time to first decision, number of splits, number of stabilized ReLUs,
LP calls, exact-gate invocations, learned-lemma reuse rate, and UNSAT-certificate export rate.
\item \textbf{Ablations:} (i) no learning, (ii) only linear lemmas, (iii) only conflict clauses,
(iv) different exact-gate thresholds, (v) different branching heuristics.
\item \textbf{Reproducibility:} configuration tables (LP solver, SMT/MILP backend, timeouts),
random seeds for tie-breaking, and exact benchmark identifiers.
\end{enumerate}

\section{Related Work (Brief)}
Complete verifiers based on SMT/MILP include Reluplex and Marabou  \cite{reluplex2017,marabouCAV19,marabouTACAS24,fastcompleteICLR21}
.
Sound-but-incomplete approaches include abstract interpretation systems (e.g., AI2) and
tight PWL domains (e.g., DeepPoly). Symbolic representation approaches (e.g., SyReNN) compute
explicit region structure for analysis tasks \cite{deeppoly2019,syrnn}
.
Our contribution is architectural: a unified incremental pipeline coupling propagation, exactness gating,
and cross-node certificate learning \cite{ai2,reluval,neurify}
.

\section{Certificate-Based Validation: Worked Example and Checker Costs}
\label{sec:validation}

To provide reproducible evidence without performance claims, we validate the core claims via checkable certificate artifacts and a fully explicit worked instance.

\subsection{Checker obligations and verification cost}
\label{subsec:checker-cost}

All pruning and learning steps in our calculus reduce to checking rational identities and inequalities.
Let a linear store be represented as $A v \le b$ with $A\in\Q^{m\times N}$ and $b\in\Q^{m}$.

\begin{proposition}[Cost of certificate checking]
\label{prop:check-cost}
Given $(A,b)$ and a purported certificate vector $\lambda\in\Q^m$, the checker for
(i) dual bound certificates (Definition~\ref{def:dual-cert}) and
(ii) Farkas infeasibility certificates (Definition~\ref{def:farkas-cert})
runs in time linear in the number of nonzeros of $A$ and the encoding size of the rationals
appearing in $(A,b,\lambda)$.
\end{proposition}

\begin{proof}
The checker evaluates $\lambda\ge 0$ componentwise, computes the products $\lambda^\top A$ and
$\lambda^\top b$ once, and verifies the required equalities/inequalities.
Each computation is a sparse matrix--vector multiplication plus scalar comparisons.
\end{proof}

Proposition~\ref{prop:check-cost} is the key non-empirical robustness hook: correctness of pruning and
learned lemmas does not depend on numerical tolerances of an LP/MILP/SMT backend, but on a deterministic
rational checker applied to exported artifacts.

\subsection{A minimal end-to-end case study (explicit certificates)}
\label{subsec:worked-example}

We give a concrete ReLU network and a safety query for which (i) propagation yields a globally reusable
bound, (ii) the negated query becomes LP-infeasible, and (iii) a short Farkas certificate closes the node.

\noindent\textbf{Network and domain..}
Let $\NNet:\R\to\R$ be defined by the following two-ReLU computation:
\[
x\in[0,1],\qquad s_1 = 2x-1,\ \ z_1=\ReLU(s_1),\qquad s_2=\tfrac12-x,\ \ z_2=\ReLU(s_2),
\qquad y = z_1 - z_2.
\]
Let the safety predicate be the linear bound $\Safety(y):\ y\le 1$.
To avoid strict inequalities in certificates, we consider the margin-violating query
\[
\neg\Safety_\varepsilon(y):\ y \ge 1+\varepsilon,\qquad \varepsilon=\tfrac{1}{10}.
\]
Refuting $\neg\Safety_\varepsilon$ implies the robust safety claim $y\le 1+\varepsilon$ on the domain.

\noindent\textbf{Propagation consequences..}
From $x\in[0,1]$ we have $s_1\in[-1,1]$ and $s_2\in[-\tfrac12,\tfrac12]$.
Hull insertion (Section~\ref{sec:prop}) therefore yields the standard envelope consequences
$0\le z_1\le 1$ and $z_2\ge 0$.
Together with the affine output constraint $y=z_1-z_2$, we obtain the valid relaxation implication
$y\le z_1\le 1$; hence $y\ge \tfrac{11}{10}$ is impossible.

\noindent\textbf{A Farkas certificate for pruning..}
Consider the reduced variable vector $v=(z_1,z_2,y)^\top$ and the following subset of node constraints:
\[
\begin{array}{rcl}
(1)& z_1 &\le 1,\\
(2)& -z_2 &\le 0,\\
(3)& -z_1 + z_2 + y &\le 0 \qquad (\text{equivalently } y\le z_1-z_2),\\
(4)& -y &\le -\tfrac{11}{10}\qquad (\text{equivalently } y\ge \tfrac{11}{10}).
\end{array}
\]
This is of the form $A v\le b$ with
\[
A=\begin{pmatrix}
 1& 0& 0\\
 0&-1& 0\\
-1& 1& 1\\
 0& 0&-1
\end{pmatrix},
\qquad
b=\begin{pmatrix}
 1\\
 0\\
 0\\
 -\tfrac{11}{10}
\end{pmatrix}.
\]
Let $\lambda=(1,1,1,1)^\top\ge 0$. Then
\[
\lambda^\top A = (1,1,1,1)
\begin{pmatrix}
 1& 0& 0\\
 0&-1& 0\\
-1& 1& 1\\
 0& 0&-1
\end{pmatrix}
= (0,0,0),
\qquad
\lambda^\top b = 1+0+0-\tfrac{11}{10} = -\tfrac{1}{10} < 0.
\]
Hence $\lambda$ is a valid Farkas infeasibility certificate (Definition~\ref{def:farkas-cert}),
so the node is closed by the Prune--Infeasible rule (Lemma~\ref{lem:prune-infeasible-sound}).

\subsection{Demonstrating Branch--Merge learning on the same instance}
\label{subsec:merge-demo}

We now illustrate the Merge rule as an actual proof step producing a parent lemma.
Split the input domain into $R_1=[0,\tfrac12]$ and $R_2=[\tfrac12,1]$.

\noindent\textbf{Child bounds (certifiable)..}
On $R_1$, $s_1=2x-1\le 0$ and $s_2=\tfrac12-x\ge 0$, so stabilization yields
$z_1=0$ and $z_2=s_2=\tfrac12-x$, hence $y=x-\tfrac12\le 0$.
On $R_2$, $s_1\ge 0$ and $s_2\le 0$, so $z_1=s_1=2x-1$ and $z_2=0$, hence $y=2x-1\le 1$.

These child bounds can be exported as dual certificates over the corresponding child linear stores.
For example, on $R_2$ the inequality $y\le 1$ follows by adding the constraints
$(y-2x\le -1)$ and $(2x\le 2)$ with multipliers $(1,1)$, which is a dual certificate instance.

\noindent\textbf{Merge into a parent lemma..}
Let $g^\top v$ select the output coordinate $y$.
From the child certificates we obtain bounds $\beta_1=0$ on $R_1$ and $\beta_2=1$ on $R_2$.
By Theorem~\ref{thm:merge}, the merged lemma $y\le \max\{\beta_1,\beta_2\}=1$ is valid on the parent region
$R=R_1\cup R_2$ and may be injected into $\Learn$ as a globally reusable constraint for this template.
This exhibits the learning rule as a \emph{compositional certificate transformation}, not a heuristic cut.

\section{Conclusion}
We presented a solver-grade, certificate-carrying architecture for PWL neural verification that
integrates incremental LP propagation, a selective exactness gate, and proof-calculus-level learning.
Sound pruning and validated counterexamples follow from a small set of independently checkable artifacts
(dual bounds, guarded cores, and Farkas certificates), while branch--merge learning lifts branch-local
certificates into reusable parent lemmas. To support reproducibility without relying on empirical claims,
Section~\ref{sec:validation} provides an explicit end-to-end worked instance together with a short
infeasibility certificate and a merge demonstration, and formalizes the cost of certificate checking
as a sparse arithmetic procedure.

% -------------------- Bibliography --------------------
% -------------------- Bibliography (25 references) --------------------
\section*{Acknowledgements}
The author gratefully acknowledges the administrative and academic support provided by the
Commissioner of Collegiate Education (CCE), Government of Andhra Pradesh, and the
Principal, Government College (Autonomous), Rajahmundry.

\section*{Statements and Declarations}

\subsection*{Funding}
The author received no external funding for this work.

\subsection*{Competing interests}
The author declare that they have no competing interests.

\subsection*{Data availability}
No new datasets were generated or analyzed in this study.

\subsection*{Code availability}
No software artifact is released with this manuscript; the algorithms are provided in pseudocode for reproducibility.

\subsection*{Author contributions}
C.~Gokavarapu: conceptualization, methodology, formal analysis, writing---original draft, writing---review \& editing.

\end{document}